\documentclass[11pt,draft,dvips]{amsart}  
\usepackage{amssymb,amsmath,amsfonts}  
\usepackage[dvips]{graphicx,color}
\usepackage{cite}
\allowdisplaybreaks[1]
\definecolor{pgray}{gray}{0.8}

\newtheorem{theorem}{Theorem}[section]
\newtheorem{cor}[theorem]{Corollary}
\newtheorem{proposition}[theorem]{Proposition}
\newtheorem{lemma}[theorem]{Lemma}

\numberwithin{equation}{section}
\title{\mbox{}}

\begin{document}
\begin{center}
{\bf \Large{
	Optimal Estimates for Far Field Asymptotics of Solutions to the Quasi-Geostrophic Equation
}}\\
\vspace{5mm}
{\sc\large
	Masakazu Yamamoto}\footnote{Graduate School of Science and Technology,
	Niigata University,
	Niigata 950-2181, Japan}\qquad
{\sc\large
	Yuusuke Sugiyama}\footnote{Department of Engineering,
	The University of Shiga Prefecture,
	Hikone 522-8533, Japan}
\end{center}
\maketitle
\vspace{-15mm}
\begin{abstract}
The initial value problem for the two dimensional dissipative quasi-geostrophic equation of the critical and the supercritical cases is considered.
Anomalous diffusion on this equation provides slow decay of solutions as the spatial parameter tends to infinity.
In this paper, uniform estimates for far field asymptotics of solutions are given.
\end{abstract}
%

\section{Introduction}
We derive far field asymptotics of solutions of the following initial-value problem:
\begin{equation}\label{qg}
\left\{
\begin{array}{lr}
	\partial_t \theta + (-\Delta)^{\alpha/2} \theta + \nabla \cdot (\theta\nabla^\bot\psi) = 0,
	&
	t > 0,~ x \in \mathbb{R}^2,\\
	(-\Delta)^{1/2}\psi = \theta,
	&
	t > 0,~ x \in \mathbb{R}^2,\\
	\theta (0,x) = \theta_0 (x),
	&
	x \in \mathbb{R}^2,
\end{array}
\right.
\end{equation}
where $\partial_t = \partial/\partial t,~ \nabla^\bot = (-\partial_2, \partial_1),~ \partial_j = \partial/\partial x_j$ for $j = 1,2$, and $(-\Delta)^{\alpha/2} \varphi = \mathcal{F}^{-1} [|\xi|^\alpha \mathcal{F} [\varphi]]$ for $0 < \alpha \le 2$.
The unknown function $\theta = \theta (t,x)$ stands for the potential temperature and $\psi = \psi (t,x)$ is the stream function (cf.\cite{Cnstntn-Mjd-Tbk}).
The fluid velocity is represented by $\nabla^\bot \psi = (-R_2 \theta, R_1 \theta)$ and $R_j = \partial_j (-\Delta)^{-1/2}$ is the Riesz transform.
When $\alpha = 1$ and $0 < \alpha < 1$, the scaling property of \eqref{qg} is in the critical and the supercritical, respectively.
In those cases, it is well-known that some smallness and smoothness for the initial-data are required to obtain the global existence of solutions in time.
Global existence of solutions in scale-invariant spaces is important problem also in the study for the Navier-Stokes flow.
Especially the critical quasi-geostrophic equation has similar stracture as the Navier-Stokes equations.
Furthermore, in the critical and the supercritical cases, the quasi-geostrophic equation seems to be elliptic and hyperbolic, respectively.
Hence the several methods for parabolic equations are not working for \eqref{qg}.
Because of those reasons, the quasi-geostrophic equation is considered by many authors (see \cite{Cffrll-Vssur,Ch-L,C-C-W,Crdb-Crdb,J05,J04,K-N-V,Mur06}).
In this paper, we treat the global solution in time which satisfies that
\begin{equation}\label{exist}
	\theta \in C \bigl( [0,\infty), L^1 (\mathbb{R}^2) \cap L^\infty (\mathbb{R}^2) \bigr).
\end{equation}
Then the mass-conservation and the uniform decay in time hold:
\begin{equation}\label{mass}
	\int_{\mathbb{R}^2} \theta (t,x) dx
	=
	\int_{\mathbb{R}^2} \theta_0 (x) dx,
	\quad
	\left\| \theta (t) \right\|_{L^\infty (\mathbb{R}^2)}
	\le
	C (1+t)^{-2/\alpha}.
\end{equation}
Those properties are confirmed initial data that is small and smooth.
In the recent paper \cite{Y-S18aX}, the smoothness and the upper bound of spatial decay of the solution are proved.
Namely, upon the condition $\theta_0 \in H^\sigma (\mathbb{R}^2)$ for $\sigma > 2$ and $|x|^2 \theta_0 \in L^q (\mathbb{R}^2)$ for $q > 2/\alpha$, the solution satisfies that
\begin{equation}\label{drbdd}
	\bigl\| (-\Delta)^{\sigma/2} \theta (t) \bigr\|_{L^2 (\mathbb{R}^2)}
	\le
	C (1+t)^{-\frac1\alpha-\frac\sigma\alpha}
\end{equation}
and
\begin{equation}\label{wttheta}
	\bigl\| |x|^2 \theta (t) \bigr\|_{L^q (\mathbb{R}^2)}
	\le
	C (1+t)^{\frac{2}{\alpha q}}.
\end{equation}
Those estimates are optimal since the fundamental solution $G_\alpha (t,x) = (2\pi)^{-1} \mathcal{F}^{-1} [e^{-t |\xi|^\alpha}] (x)$ of the linear equation $\partial_t \theta + (-\Delta)^{\alpha/2} \theta = 0$ fulfills that $\| (-\Delta)^{\sigma/2} G_\alpha (t) \|_{L^2 (\mathbb{R}^2)} = t^{-\frac1\alpha-\frac\sigma\alpha} \| (-\Delta)^{\sigma/2} G_\alpha (1) \|_{L^2 (\mathbb{R}^2)}$ and $\| |x|^2 G_\alpha (t) \|_{L^q (\mathbb{R}^2)}$ $= t^{\frac{2}{\alpha q}} \| |x|^2 G_\alpha (1) \|_{L^q (\mathbb{R}^2)}$.
When $\alpha = 1$, this fundamental solution is given by the Poisson kernel $P$.
Moreover the lower bound of spatial decay is derived in \cite{Y-S18aX}:
\begin{equation}\label{L2}
	\bigl\| |x|^2 \bigl( \theta (t) - M G_\alpha (t) \bigr) \bigr\|_{L^2 (\mathbb{R}^2)}
	\le
	C L_\alpha (t) \bigl( \log (2+t) \bigr)^{1/2},
\end{equation}
where $M = \int_{\mathbb{R}^2} \theta_0 (x) dx$, and $L_\alpha (t) = \log (2+t)$ for $\alpha = 1$ and $L_\alpha (t) = 1$ for $0 < \alpha < 1$.
Since $\| |x|^2 G_\alpha (t) \|_{L^2 (\mathbb{R}^2)} = + \infty$ and $G_\alpha \in C ((0,\infty) \times \mathbb{R}^2)$, \eqref{L2} intends that $\theta$ and $M G_\alpha$ are canceled in far field.
Therefore the asymptotic profile of $\theta$ as $|x| \to +\infty$ is presented by $MG_\alpha$.
This idea is developed from pointwise-estimates for Navier-Stokes flow via Miyakawa \cite{Mykw00}, and Miyakawa and Schonbek \cite{Mykw-Schnbk}, and firstly applied by Brandolese \cite{Brndls}, and Brandolese and Vigneron \cite{Brndls-Vgnrn}.
A main goal of this paper is to show the uniform estimate of the spatial decay of the solution.
Specifically, we provide the similar estimate as \eqref{L2} in $L^\infty (\mathbb{R}^2)$.
For solutions to the fractional diffusion equation $\partial_t u + (-\Delta)^{\alpha/2} u = \nabla \cdot f(u)$ with $1 < \alpha < 2$ and some suitable $f$, general theory of spatial decay is given by Brandolese and Karch \cite{Brndls-Krch}.
This theory is based on the $L^p$-$L^q$ estimates for $e^{-t(-\Delta)^{\alpha/2}}$ and available for \eqref{qg} in the subcritical case $1 < \alpha < 2$ since \eqref{qg} is parabolic in this case.
However, since the nonlinearity balances to the dissipation in $\alpha = 1$, the general theory does not work in the case that $0 < \alpha \le 1$.
More precisely, it is diffucult to estimate the nonlinear term in the integral equation $\int_0^t \nabla e^{-(t-s)(-\Delta)^{\alpha/2}} \cdot f(u) ds$ in usual way, since $\nabla e^{-t(-\Delta)^{\alpha/2}}$ is not integrable near $t = 0$, which requires to estimate $f(u)$ in the weighted Sobolev spaces with some positive differential order.
In particular, we estimate $\| |x|^2 \nabla \theta (t) \|_{L^p (\mathbb{R}^2)}$ by using the energy method (cf.\cite{Y-S-NA}).
Furthermore we prepare some uniform decay properties for $G_\alpha$ (see Lemmas \ref{prop-wt} and \ref{lem1}).
We note that the Moser-Nash iteration method is not used, which is usually employed in order to obtain $L^\infty$ estimate from $L^p$ (e.g.\cite{E-Z,K-K,Og-Y}).
In \cite{Y-S18aX}, the estimate \eqref{L2} is derived by the energy method.
The proof of the uniform estimate \eqref{astn} in the main theorem is based on the $L^p$-$L^q$ argument.
We can express the spatial decay of the solution both in the critical and the supercritical cases.

\begin{theorem}\label{thm}
Let $0 < \alpha \le 1,~ p > 2/\alpha,~ (1+|x|) \theta_0 \in L^1 (\mathbb{R}^2),~ |x|^{3+\alpha} \theta_0 \in L^\infty (\mathbb{R}^2),~ \theta_0 \in H^3 (\mathbb{R}^2)$ and $|x|^2 \nabla \theta_0 \in L^p (\mathbb{R}^2)$.
Assume that $\| \theta_0 \|_{H^3 (\mathbb{R}^2)}$ is sufficiently small, and the solution $\theta$ fulfills \eqref{exist} and \eqref{mass}.
Then
\begin{equation}\label{astn}
	\bigl\| |x|^{3+\alpha} \bigl( \theta (t) - MG_\alpha (t) \bigr) \bigr\|_{L^\infty (\mathbb{R}^2)}
	\le
	C (1+t),
\end{equation}
where $M = \int_{\mathbb{R}^2} \theta_0 (x) dx$.
\end{theorem}
We emphasize that $\| |x|^{3+\alpha} G_\alpha (t) \|_{L^\infty (\mathbb{R}^2)} = + \infty$.
Thus Theorem \ref{thm} states that $\theta$ and $MG_\alpha$ are canceled uniformly in far field.
We remark that the assertions in this theorem is sharp in time.
Indeed, \eqref{astn} for $G_\alpha (t) * \theta_0$ instead of $\theta$ is fulfilled (see Lemma \ref{lem1} in Section 2), and the scaling property of $G_\alpha$ guarantees the sharpness of those estimates.
Namely, from the mean value theorem, we expect that the decay-rate of the top term of $G_\alpha (t) * \theta_0 - M G_\alpha (t)$ is given by one of $\nabla G_\alpha (t)$, and $\| |x|^{3+\alpha} \nabla G_\alpha (t) \|_{L^\infty (\mathbb{R}^2)} = t \| |x|^{3+\alpha} \nabla G_\alpha (1) \|_{L^\infty (\mathbb{R}^2)}$ and $\| |x|^{3+\alpha} \nabla G_\alpha (1) \|_{L^\infty (\mathbb{R}^2)} < +\infty$.
The details are in the proof of Lemma \ref{lem1}.
Furthermore \eqref{astn} is optimal also in $x$ since $\| |x|^{3+\alpha+\varepsilon} \nabla G_\alpha (t) \|_{L^\infty (\mathbb{R}^2)} = + \infty$ for any $\varepsilon > 0$.
A coupling of this theorem and the property of $G_\alpha$ provides the obvious decay as follows.
\begin{cor}\label{cor}
Let $\varrho = \varrho(x) > 0$ be radially symmetric and monotone increasing in $|x|$, and satisfy that $\varrho (x) \to + \infty$ as $|x| \to +\infty$.
Let $0 < \alpha \le 1,~ p > 2/\alpha,~ \theta_0 \in L^1 (\mathbb{R}^2) \cap H^3 (\mathbb{R}^2),~ |x|^{2+\alpha} \varrho\, \theta_0 \in L^\infty (\mathbb{R}^2)$ and $|x|^2 \nabla \theta_0 \in L^p (\mathbb{R}^2)$.
Assume that the solution $\theta$ fulfills \eqref{exist}, \eqref{mass}, and \eqref{drbdd} for $\sigma = 3$.
Then, for any fixed $t>0$,
\[
	|x|^{2+\alpha} \theta (t,x)
	\to C_\alpha M t
\]
as $|x| \to +\infty$, where $M = \int_{\mathbb{R}^2} \theta_0 (x) dx$ and $C_\alpha = \alpha 2^{\alpha- 1} \pi^{-2} \sin \tfrac{\alpha \pi}2 \Gamma (1+\tfrac{\alpha}2) \Gamma (\tfrac\alpha{2})$.
\end{cor}
We should remark that $\theta$ on Theorem \ref{thm} and Corollary \ref{cor} fulfill \eqref{wttheta} for any $q > 2/\alpha$.
Indeed, if $|x|^2 \theta_0 \in L^q (\mathbb{R}^2)$, then Proposition \ref{prop-wt} in Section 2 yields \eqref{wttheta}, and we confirm that
\[
	\bigl\| |x|^2 \theta_0 \bigr\|_{L^q (\mathbb{R}^2)}
	\le
	\bigl\| (1+ |x|)^{-\alpha} \bigr\|_{L^q (\mathbb{R}^2)} \bigl\| |x|^2 (1+|x|)^{\alpha} \theta_0 \bigr\|_{L^\infty (\mathbb{R}^2)}
	< + \infty.
\]

\noindent
{\bf Notation.}
The Fourier transform and its inverse are defined by $\mathcal{F} [\varphi] (\xi) = (2\pi)^{-1} \int_{\mathbb{R}^2} e^{-ix\cdot\xi} \varphi (x) dx$ and $\mathcal{F}^{-1} [\varphi] (x) = (2\pi)^{-1} \int_{\mathbb{R}^2} e^{ix\cdot\xi} \varphi (\xi) d\xi$, where $i = \sqrt{-1}$.
The derivations are abbreviated by $\partial_t = \partial / \partial t,~ \partial_j = \partial / \partial x_j$ for $j = 1,2,~ \nabla = (\partial_1,\partial_2),~ \nabla^\bot = (-\partial_2,\partial_1)$ and $\Delta = \partial_1^2 + \partial_2^2$.
The fractional Laplacian and its inverse, and the Riesz transform are defined by $(-\Delta)^{\alpha/2} \varphi = \mathcal{F}^{-1} [|\xi|^\alpha \mathcal{F} [\varphi]],~ (-\Delta)^{-\sigma/2} \varphi = \mathcal{F}^{-1} [|\xi|^{-\sigma} \mathcal{F} [\varphi]]$ for $0 < \sigma < 2$, and $R_j \varphi = \partial_j (-\Delta)^{-1/2} \varphi = \mathcal{F}^{-1} [i\xi_j |\xi|^{-1} \mathcal{F} [\varphi]]$ for $j = 1,2$, respectively.
The H\"older conjugate of $1 \le p \le \infty$ is denoted by $p'$, i.e., $\frac1p + \frac1{p'} = 1$.
For $\beta = (\beta_1,\beta_2) \in \mathbb{Z}_+^2 = (\mathbb{N} \cup \{ 0 \})^2,~ |\beta| = \beta_1 + \beta_2$.
For some operators $A$ and $B,~ [A,B] = AB - BA$.
Various positive constants and suitable fuctions are denoted by $C$ and $\varphi$, respectively.

\section{Preliminaries}
The Duhamel principle yields that
\begin{equation}\label{MS}
	\theta (t)
	=
	G_\alpha (t) * \theta_0
	-
	\int_0^t
		\nabla G_\alpha (t-s) * (\theta \nabla^\bot \psi) (s)
	ds,
\end{equation}
where $G_\alpha (t) = (2\pi)^{-1} \mathcal{F}^{-1} [e^{-t|\xi|^\alpha}]$.
It is well-known that
\begin{equation}\label{scG}
	G_\alpha (t,x) = t^{-2/\alpha} G_\alpha (1,t^{-1/\alpha} x)
\end{equation}
for $(t,x) \in (0,+\infty) \times \mathbb{R}^2$, and
\begin{equation}\label{decayG}
	\bigl| \nabla^\beta G_\alpha (1,x) \bigr|
	\le
	C_\beta (1+|x|^2)^{-1-\frac\alpha{2}-\frac{|\beta|}2}
\end{equation}
for $\beta \in \mathbb{Z}_+^2$ and $x \in \mathbb{R}^2$ (see\cite{Klkltsv,Y-S18aX}).
The spatial decay of $G_\alpha$ is published as the following (cf.\cite{Blmntr-Gtr}):
\begin{equation}\label{B-G}
	|x|^{2+\alpha} G_\alpha (t,x) \to C_\alpha t
\end{equation}
as $|x| \to +\infty$, where $C_\alpha$ is introduced in Corollary \ref{cor}.
The following lemma plays a crucial role in the energy estimates.
\begin{lemma}[Stroock-Varopoulos inequality \cite{Crdb-Crdb,J05,L-S}]\label{lemC-C}
	Let $0 \le \alpha \le 2,~ q \ge 2$ and $f \in W^{\alpha,q} (\mathbb{R}^2)$.
	Then
	\[
		\int_{\mathbb{R}^2}
			|f|^{q-2} f (-\Delta)^{\alpha/2} f
		dx
		\ge
		\frac2q \int_{\mathbb{R}^2}
			\left| (-\Delta)^{\alpha/4} (|f|^{q/2}) \right|^2
		dx
	\]
	holds.
\end{lemma}
We need the following inequalities of Sobolev type.
\begin{lemma}[Hardy-Littlewood-Sobolev's inequality \cite{S,Z}]\label{HLS}~
	Let $0 < \sigma < 2,~ 1 < p < \frac{2}{\sigma}$ and $\frac1{p_*} = \frac1p - \frac{\sigma}2$.
	Then there exists a positive constant $C$ such that
	\[
		\bigl\| (-\Delta)^{-\sigma/2} \varphi \bigr\|_{L^{p_*} (\mathbb{R}^2)}
		\le
		C \bigl\| \varphi \bigr\|_{L^p (\mathbb{R}^2)}
	\]
	for any $\varphi \in L^p (\mathbb{R}^2)$.
\end{lemma}
\begin{lemma}[Gagliardo-Nirenberg inequality \cite{H-Y-Z,K-PP,M-N-S-S}]\label{GN}
	Let $0 < \sigma < s < 2,~ 1 < p_1, p_2 < \infty$ and $\frac1p = (1-\frac\sigma{s}) \frac1{p_1} + \frac\sigma{s} \frac1{p_2}$.
	Then
	\[
		\bigl\| (-\Delta)^{\sigma/2} \varphi \bigr\|_{L^p (\mathbb{R}^2)}
		\le
		C \bigl\| \varphi \bigr\|_{L^{p_1} (\mathbb{R}^2)}^{1-\frac\sigma{s}}
		\bigl\| (-\Delta)^{s/2} \varphi \bigr\|_{L^{p_2} (\mathbb{R}^2)}^{\frac\sigma{s}}
	\]
	holds.
\end{lemma}
To care the Riesz transforms, we call the following H\"ormander-Mikhlin type estimate.
\begin{lemma}[H\"ormander-Mikhlin inequality\,\cite{Hrmndr,Mkhln,S-S}]\label{lem-HM}
Let $N \in \mathbb{Z}_+,~ 0 < \mu \le 1$ and $\lambda = N + \mu - 2$.
Assume that $\varphi \in C^\infty (\mathbb{R}^2\backslash \{ 0 \})$ satisfies the following conditions:
\begin{itemize}
\item
	$\nabla^\gamma \varphi \in L^1 (\mathbb{R}^2)$ for any $\gamma \in \mathbb{Z}_+^2$ with $|\gamma| \le N$;
\item
	$|\nabla^\gamma \varphi (\xi)| \le C_\gamma |\xi|^{\lambda - |\gamma|}$ for $\xi \neq 0$ and $\gamma \in \mathbb{Z}_+^2$ with $|\gamma| \le N+2$.
\end{itemize}
Then
\[
	\sup_{x \neq 0} \bigl( |x|^{2+\lambda} \bigl| \mathcal{F}^{-1} [\varphi] (x) \bigr| \bigr)
	< + \infty
\]
holds.
\end{lemma}
For the details of this Lemma, see \cite{S-S}.
The following proposition is confirmed in \cite{Crdb-Crdb,J04,Y-S18aX}.
\begin{proposition}\label{tr}
Let $\sigma > 2,~ \theta_0 \in H^\sigma (\mathbb{R}^2)$ and $\| \theta_0 \|_{H^\sigma (\mathbb{R}^2)}$ be small.
Assume that the solution $\theta$ satisfies \eqref{exist} and \eqref{mass}.
Then \eqref{drbdd} holds.
\end{proposition}
The authors proved the following proposition in \cite{Y-S18aX}.
\begin{proposition}\label{prop-wt}
Let $q > 2/\alpha$ and $|x|^2 \theta_0 \in L^q (\mathbb{R}^2)$.
Assume that the solution $\theta$ of \eqref{qg} satiefies \eqref{exist} and \eqref{mass}.
Then \eqref{wttheta} holds.
\end{proposition}
The term of initial-data on \eqref{MS} satisfies the following lemma.
\begin{lemma}\label{lem1}
Let $(1+|x|) \theta_0 \in L^1 (\mathbb{R}^2)$ and $|x|^{3+\alpha} \theta_0 \in L^\infty (\mathbb{R}^2)$.
Then
\[
	\bigl\| |x|^{3+\alpha} \bigl( G_\alpha (t) * \theta_0 - M G_\alpha (t) \bigr) \bigr\|_{L^\infty (\mathbb{R}^2)}
	\le C \bigl( t \bigl\| |x| \theta_0 \bigr\|_{L^1 (\mathbb{R}^2)} + \bigl\| |x|^{3+\alpha} \theta_0 \bigr\|_{L^\infty (\mathbb{R}^2)} \bigr)
\]
for $t > 0$, where $M = \int_{\mathbb{R}^2} \theta_0 (x) dx$.
\end{lemma}
\begin{proof}
The mean value theorem gives that
\begin{equation}\label{meanG}
\begin{split}
	G_\alpha (t) * \theta_0 - M G_\alpha (t)
	&=
	\int_{|y| > |x|/2} \bigl( G_\alpha (t,x-y) - G_\alpha (t,x) \bigr) \theta_0 (y) dy\\
	&+
	\int_{|y| \le |x|/2} \int_0^1 (-y\cdot\nabla) G_\alpha (t,x-\lambda y) \theta_0 (y) d\lambda dy.
\end{split}
\end{equation}
The first term fulfills that
\[
\begin{split}
	&\biggl| |x|^{3+\alpha} \int_{|y| > |x|/2} \bigl( G_\alpha (t,x-y) - G_\alpha (t,x) \bigr) \theta_0 (y) dy \biggr|\\
	&\le
	C \bigl\| G_\alpha (t) \bigr\|_{L^1 (\mathbb{R}^2)} \bigl\| |x|^{3+\alpha} \theta_0 \bigr\|_{L^\infty (\mathbb{R}^2)}
	+
	C \bigl\| |x|^{2+\alpha} G_\alpha (t) \bigr\|_{L^\infty (\mathbb{R}^2)} \bigl\| |x| \theta_0 \bigr\|_{L^1 (\mathbb{R}^2)}\\
	&\le
	 C \bigl( t \bigl\| |x| \theta_0 \bigr\|_{L^1 (\mathbb{R}^2)} + \bigl\| |x|^{3+\alpha} \theta_0 \bigr\|_{L^\infty (\mathbb{R}^2)} \bigr).
\end{split}
\]
From \eqref{scG} and \eqref{decayG}, for $|y| \le |x|/2$ and $0 < \lambda < 1$, we see that $|x| \le 2|x-\lambda y|$ and then
\[
\begin{split}
	|x|^{3+\alpha} \bigl| \nabla G_\alpha (t,x-\lambda y) \bigr|
	&=
	t^{-3/\alpha} |x|^{3+\alpha} \bigl| \nabla G_\alpha (1,t^{-1/\alpha}(x-\lambda y)) \bigr|
	\le
	C t.
\end{split}
\]
Thus, we have that
\[
	\biggl| |x|^{3+\alpha} \int_{|y| \le |x|/2} \int_0^1 (-y\cdot\nabla) G_\alpha (t,x-\lambda y) \theta_0 (y) d\lambda dy \biggr|
	\le
	C t \bigl\| |x| \theta_0 \bigr\|_{L^1 (\mathbb{R}^2)}
\]
and conclude the proof.
\end{proof}
\begin{lemma}\label{lem2}
Let $\theta_0 \in L^1 (\mathbb{R}^2)$ and $|x|^{2+\alpha} \varrho\, \theta_0 \in L^\infty (\mathbb{R}^2)$, where $\varrho = \varrho (x)$ is defined in Corollary \ref{cor}.
Then, for any $t > 0$,
\[
	|x|^{2+\alpha} \bigl( G_\alpha (t) * \theta_0 - M G_\alpha (t) \bigr)
	\to 0
\]
as $|x| \to +\infty$, where $M = \int_{\mathbb{R}^2} \theta_0 (x) dx$.
\end{lemma}
\begin{proof}
We use \eqref{meanG}.
For the first term, we see that
\[
\begin{split}
	&|x|^{2+\alpha} \biggl| \int_{|y| > |x|/2} \bigl( G_\alpha (t,x-y) - G_\alpha (t,x) \bigr) \theta_0 (y) dy \biggr|\\
	&\le
	C \varrho (\tfrac{x}2)^{-1} \int_{|y| > |x|/2} \bigl| G_\alpha (t,x-y) \bigr| \bigl| |y|^{2+\alpha} \varrho(y) \theta_0 (y) \bigr| dy
	+ C \int_{|y| > |x|/2} \bigl| |x|^{2+\alpha} G_\alpha (t,x) \bigr| \bigl| \theta_0 (y) \bigr| dy\\
	&\le
	C \varrho (\tfrac{x}2)^{-1} \bigl\| |x|^{2+\alpha} \varrho \theta_0 \bigr\|_{L^\infty (\mathbb{R}^2)}
	+
	C \int_{|y| > |x|/2} \bigl| \theta_0 (y) \bigr| dy
	\to 0
\end{split}
\]
as $|x| \to +\infty$.
For the second term on \eqref{meanG}, we choose sufficiently small $\varepsilon > 0$ and $\frac2{1+\alpha} < r < \frac2{1+\varepsilon}$, then we have by the similar argument as in the proof of Lemma \ref{lem1} that
\[
\begin{split}
	&|x|^{2+\alpha} \biggl| \int_{|y| \le |x|/2} \int_0^1 (-y\cdot\nabla) G_\alpha (t,x-\lambda y) \theta_0 (y) d\lambda dy \biggr|\\
	&\le
	C |x|^{-\varepsilon} \int_{|y| \le |x|/2} \int_0^1 |x-\lambda y|^{2+\alpha+\varepsilon} \bigl| \nabla G_\alpha (t,x-\lambda y) \bigr|
	\bigl| y \theta_0 (y) \bigr| d\lambda dy\\
	&\le
	C |x|^{-\varepsilon} t^{-\frac2{\alpha r} + \frac1\alpha + 1 + \frac\varepsilon\alpha} \bigl\| x \theta_0 \bigr\|_{L^r (\mathbb{R}^2)}
	\to 0
\end{split}
\]
as $|x| \to +\infty$.
Here we remark that $-\frac2{\alpha r} + \frac1\alpha + 1 + \frac\varepsilon\alpha > 0$ and
\[
	\bigl\| x \theta_0 \bigr\|_{L^r (\mathbb{R}^2)}
	\le
	 \bigl\| (1+|x|)^{-1-\alpha} \bigr\|_{L^r (\mathbb{R}^2)} \bigl\| |x| (1+|x|)^{1+\alpha} \theta_0 \bigr\|_{L^\infty (\mathbb{R}^2)}
	< + \infty.
\]
Hence we complete the proof.
\end{proof}
\section{Proof of main theorems}
To show our main assertions, we prepare the estimate for the nonlinear effect.
We denote the nonlinear term on \eqref{MS} by $v$, i.e.,
\begin{equation}\label{def-v}
	v (t) = 
	-
	\int_0^t
		\nabla G_\alpha (t-s) * (\theta \nabla^\bot \psi) (s)
	ds.
\end{equation}
Then decay-rate of $v$ as $|x| \to +\infty$ is published as follows.
\begin{proposition}\label{prop-v2}
Let $0 < \alpha \le 1,~ p > 2/\alpha,~ \theta_0 \in L^1 (\mathbb{R}^2) \cap H^3 (\mathbb{R}^2),~ |x|^{2+\alpha} \theta_0 \in L^\infty (\mathbb{R}^2)$ and $|x|^2 \nabla \theta_0 \in L^p (\mathbb{R}^2)$.
Assume that the solution $\theta$ fulfills \eqref{mass} and \eqref{drbdd} for $\sigma = 3$.
Then $v$ defined by \eqref{def-v} satisfies that
$\| |x|^{3+\alpha} v(t) \|_{L^\infty (\mathbb{R}^2)} \le C (1+t)$.
\end{proposition}
\begin{proof}
From the definition, we see for $j = 1,2$ that
\[
\begin{split}
	x_j v &= - \int_0^t (x_j \nabla G_\alpha) (t-s) * (\theta \nabla^\bot\psi) (s) ds\\
	&- \int_0^t G_\alpha (t-s) * (\theta(\nabla^\bot\psi)_j) (s) ds
	- \int_0^t G_\alpha (t-s) * (x_j \nabla\theta\cdot\nabla^\bot\psi) (s) ds,
\end{split}
\]
where $(\nabla^\bot\psi)_j$ is the $j$-th component of $\nabla^\bot\psi$.
Hence
\begin{equation}\label{bss}
\begin{split}
	&\bigl\| |x|^{2+\alpha} x_j v \bigr\|_{L^\infty (\mathbb{R}^2)}
	\le
	C \int_0^t
		\bigl\| |x|^{2+\alpha} x_j \nabla G_\alpha (t-s) \bigr\|_{L^\infty (\mathbb{R}^2)}
		\bigl\| \theta \nabla^\bot \psi \bigr\|_{L^1 (\mathbb{R}^2)}
	ds\\
	&+
	C \int_0^t \int_{\mathbb{R}^2}
		\bigl|  (x_j-y_j) \nabla G_\alpha (t-s,x-y) \bigr|
		\bigl| |y|^{2+\alpha} (\theta \nabla^\bot \psi)(s,y) \bigr|
	dyds\\
	&+
	\int_0^t
		\bigl\| |x|^{2+\alpha} G_\alpha (t-s) \bigr\|_{L^\infty (\mathbb{R}^2)}
		\bigl\| \theta (\nabla^\bot\psi)_j (s) \bigr\|_{L^1 (\mathbb{R}^2)}
	ds\\
	&+
	\int_0^t \int_{\mathbb{R}^2}
		\bigl| G_\alpha (t-s,x-y) \bigr|
		\bigl| |y|^{2+\alpha} (\theta (\nabla^\bot\psi)_j) (s,y) \bigr|
	dyds\\
	&+
	\int_0^t
		\bigl\| |x|^{2+\alpha} G_\alpha (t-s) \bigr\|_{L^\infty (\mathbb{R}^2)}
		\bigl\| x_j (\nabla \theta \cdot \nabla^\bot\psi) (s) \bigr\|_{L^1 (\mathbb{R}^2)}
	ds\\
	&+
	\int_0^t \int_{\mathbb{R}^2}
		\bigl| G_\alpha (t-s,x-y) \bigr|
		\bigl| |y|^{2+\alpha} y_j (\nabla \theta \cdot \nabla^\bot\psi) (s,y) \bigr|
	dyds.
\end{split}
\end{equation}
From \eqref{mass}-\eqref{wttheta}, \eqref{scG} and \eqref{decayG}, we see that the first and the third terms are bounded by $C (1+t)$.
We estimate the second, the fourth and the fifth terms later.
To care the last term, we prepare the estimate for $\| |x|^2 \nabla \theta \|_{L^p (\mathbb{R}^2)}$.
We derivate the first equality on \eqref{qg} in $x_k$ and multiply $|x|^{2p} |\partial_k \theta|^{p-2} \partial_k \theta$.
Then, by integrating it in $(0,t) \times \mathbb{R}^2$ and employing Lemma \ref{lemC-C}, we see for the second and the last terms that
\[
\begin{split}
	&\int_{\mathbb{R}^2}
		|x|^{2p} |\partial_k \theta|^{p-2} \partial_k \theta (-\Delta)^{\alpha/2} \partial_k \theta
	dx\\
	&=
	\int_{\mathbb{R}^2}
		\bigl| |x|^2 \partial_k \theta \bigr|^{p-2} |x|^2 \partial_k \theta (-\Delta)^{\alpha/2} (|x|^2 \partial_k \theta)
	dx
	-
	\int_{\mathbb{R}^2}
		\bigl| |x|^2 \partial_k \theta \bigr|^{p-2} |x|^2 \partial_k \theta \bigl[ (-\Delta)^{\alpha/2}, |x|^2 \bigr] \partial_k \theta
	dx\\
	&\ge
	\frac2p \bigl\| (-\Delta)^{\alpha/4} \bigl( \bigl| |x|^2 \partial_k \theta \bigr|^{p/2} \bigr) \bigr\|_{L^2 (\mathbb{R}^2)}^2
	-
	\int_{\mathbb{R}^2}
		\bigl| |x|^2 \partial_k \theta \bigr|^{p-2} |x|^2 \partial_k \theta \bigl[ (-\Delta)^{\alpha/2}, |x|^2 \bigr] \partial_k \theta
	dx
\end{split}
\]
and
\[
\begin{split}
	&\int_{\mathbb{R}^2}
		|x|^{2p} |\partial_k \theta|^{p-2} \partial_k \theta \partial_k \nabla \cdot (\theta \nabla^\bot\psi)
	dx\\
	&=
	-2p \int_{\mathbb{R}^2}
		\bigl| |x|^2 \partial_k \theta \bigr|^{p-2} |x|^2 \partial_k \theta x_k \nabla\theta \cdot \nabla^\bot \psi
	dx
	-(p-1) \int_{\mathbb{R}^2}
		\bigl| |x|^2 \partial_k \theta \bigr|^{p-2} \partial_k^2 \theta |x|^4 \nabla\theta\cdot\nabla^\bot\psi
	dx.
\end{split}
\]
Here we used the relation $\nabla \cdot (\theta \nabla^\bot\psi) = \nabla \theta \cdot \nabla^\bot \psi$.
Therefore we have that
\begin{equation}\label{bsg}
\begin{split}
	&\frac1p \sum_{k=1}^2 \bigl\| |x|^2 \partial_k \theta \bigr\|_{L^p (\mathbb{R}^2)}^p
	+
	\frac2p \sum_{k=1}^2 \int_0^t \bigl\| (-\Delta)^{\alpha/4} \bigl( \bigl| |x|^2 \partial_k \theta \bigr|^{p/2} \bigr) \bigr\|_{L^2 (\mathbb{R}^2)}^2 ds\\
	&\le
	\frac1p \sum_{k=1}^2 \bigl\| |x|^2 \partial_k \theta_0 \bigr\|_{L^p (\mathbb{R}^2)}^p
	+
	2p \sum_{k=1}^2 \int_0^t \int_{\mathbb{R}^2}
		\bigl| |x|^2 \partial_k \theta \bigr|^{p-2} |x|^2 \partial_k \theta x_k \nabla\theta \cdot \nabla^\bot \psi
	dxds\\
	&+
	(p-1) \sum_{k=1}^2 \int_0^t \int_{\mathbb{R}^2}
		\bigl| |x|^2 \partial_k \theta \bigr|^{p-2} \partial_k^2 \theta |x|^4 \nabla\theta\cdot\nabla^\bot\psi
	dxds\\
	&+
	\sum_{k=1}^2 \int_0^t \int_{\mathbb{R}^2}
		\bigl| |x|^2 \partial_k \theta \bigr|^{p-2} |x|^2 \partial_k \theta \bigl[ (-\Delta)^{\alpha/2}, |x|^2 \bigr] \partial_k \theta
	dxds.
\end{split}
\end{equation}
From \eqref{MS} and Hausdorf-Young's inequality, we see that
\[
\begin{split}
	\bigl\| |x|^2 R_l \theta \bigr\|_{L^\infty (\mathbb{R}^2)}
	&\le
	\bigl\| |x|^2 (R_l G_\alpha (t) * \theta_0) \bigr\|_{L^\infty (\mathbb{R}^2)}
	+
	\biggl\| |x|^2 \int_0^t R_l G_\alpha (t-s) * (\nabla\theta \cdot \nabla^\bot \psi) (s) ds \biggr\|_{L^\infty (\mathbb{R}^2)}\\
	&\le
	C \Bigl( 1 + \bigl\| \theta_0 \bigr\|_{H^2 (\mathbb{R}^2)} \bigl\| |x|^2 \nabla^\bot \psi \bigr\|_{L^\infty ((0,\infty)\times\mathbb{R}^2)} \Bigr).
\end{split}
\]
Here we used the scaling property $\| |x|^2 R_l G_\alpha (t) \|_{L^\infty (\mathbb{R}^2)} = \| |x|^2 R_l G_\alpha (1) \|_{L^\infty (\mathbb{R}^2)}$.
We remark that Lemma \ref{lem-HM} guarantees that $|x|^2 R_l G_\alpha (1) \in L^\infty (\mathbb{R}^2)$.
The relation $\nabla^\bot \psi = (-R_2 \theta, R_1 \theta)$ and the smallness of $\theta_0$ conclude that $\| |x|^2 \nabla^\bot \psi \|_{L^\infty (\mathbb{R}^2)} \le C$.
Hence
for the second term on \eqref{bsg} and some small $\delta > 0$, we have that
\begin{equation}\label{bsg2}
\begin{split}
	&\sum_{k=1}^2 \biggl| \int_0^t \int_{\mathbb{R}^2}
		\bigl| |x|^2 \partial_k \theta \bigr|^{p-2} |x|^2 \partial_k \theta x_k \nabla\theta \cdot \nabla^\bot \psi
	dxds \biggr|\\
	&\le
	C \sum_{k=1}^2 \int_0^t
		\bigl\| |x|^2 \partial_k \theta \bigr\|_{L^p (\mathbb{R}^2)}^{p-1}
		\bigl\| |x|^2  \nabla^\bot \psi \bigr\|_{L^\infty (\mathbb{R}^2)}^{1/2}
		\bigl\| |\nabla^\bot\psi|^{1/2} \nabla\theta \bigr\|_{L^p (\mathbb{R}^2)}
	ds\\
	&\le
	C \sum_{k=1}^2 \bigl\| |x|^2 \partial_k \theta \bigr\|_{L^\infty (0,t: L^p (\mathbb{R}^2))}^{p-1}
	 \int_0^t
		(1+s)^{-\frac2\alpha (1-\frac1p)-\frac2\alpha}
	ds
	\le
	C_\delta + \delta \sum_{k=1}^2 \bigl\| |x|^2 \partial_k \theta \bigr\|_{L^\infty (0,t: L^p (\mathbb{R}^2))}^p.
\end{split}
\end{equation}
For the third term on \eqref{bsg}, we see for $\frac1r = \frac1p + \frac\alpha{2} - \frac\alpha{2p}$ that
\[
\begin{split}
	&\sum_{k=1}^2 \biggl| \int_0^t \int_{\mathbb{R}^2}
		\bigl| |x|^2 \partial_k \theta \bigr|^{p-2} \partial_k^2 \theta |x|^4 \nabla\theta\cdot\nabla^\bot\psi
	dxds \biggr|\\
	&\le
	C \sum_{k,l=1}^2\int_0^t
		\bigl\| (-\Delta)^{\alpha/4} \bigl( \bigl| |x|^2 \partial_k \theta \bigr|^{p/2} \bigr) \bigr\|_{L^2 (\mathbb{R}^2)}^{2/p'}
		\bigl\| \partial_l^2 \theta \bigr\|_{L^r (\mathbb{R}^2)}
	ds.
\end{split}
\]
Here we used the Sobolev inequality $\| \varphi \|_{L^{2_*} (\mathbb{R}^2)} \le C \| (-\Delta)^{\alpha/4} \varphi \|_{L^2 (\mathbb{R}^2)}$ for $\frac1{2_*} = \frac12 - \frac\alpha{4}$.
Lemma \ref{GN} with Proposition \ref{tr} guarantees that $\partial_l^2 \theta \in L^r (\mathbb{R}^2)$ and $\| \partial_l^2 \theta \|_{L^r (\mathbb{R}^2)} \le C (1+t)^{-\frac2\alpha (1-\frac1r) - \frac2\alpha}$.
Hence
\begin{equation}\label{bsg3}
\begin{split}
	&\sum_{k=1}^2 \biggl| \int_0^t \int_{\mathbb{R}^2}
		\bigl| |x|^2 \partial_k \theta \bigr|^{p-2} \partial_k^2 \theta |x|^4 \nabla\theta\cdot\nabla^\bot\psi
	dxds \biggr|\\
	&\le
	C_\delta \int_0^t (1+s)^{-\frac2\alpha (p-1) - \frac{2p}\alpha -1 + p} ds
	+
	\delta \sum_{k=1}^2\int_0^t
		\bigl\| (-\Delta)^{\alpha/4} \bigl( \bigl| |x|^2 \partial_k \theta \bigr|^{p/2} \bigr) \bigr\|_{L^2 (\mathbb{R}^2)}^2
	ds\\
	&\le
	C_\delta
	+
	\delta \sum_{k=1}^2\int_0^t
		\bigl\| (-\Delta)^{\alpha/4} \bigl( \bigl| |x|^2 \partial_k \theta \bigr|^{p/2} \bigr) \bigr\|_{L^2 (\mathbb{R}^2)}^2
	ds.
\end{split}
\end{equation}
Calculus on the Fourier symbol yield that $[ (-\Delta)^{\alpha/2}, |x|^2 ] \partial_k \theta = \alpha^2 (-\Delta)^{\frac{\alpha-2}2} \partial_k \theta - 2\alpha \nabla (-\Delta)^{\frac{\alpha-2}2} \cdot (x \partial_k \theta)$.
Thus, for $\frac1r = \frac1p + \frac{1-\alpha}2$,
\[
\begin{split}
	&\sum_{k=1}^2 \biggl|
		\int_0^t \int_{\mathbb{R}^2}
		\bigl| |x|^2 \partial_k \theta \bigr|^{p-2} |x|^2 \partial_k \theta \bigl[ (-\Delta)^{\alpha/2}, |x|^2 \bigr] \partial_k \theta
	dxds \biggr|\\
	&\le
	C \sum_{k=1}^2 \bigl\| |x|^2 \partial_k \theta \bigr\|_{L^\infty (0,t: L^p (\mathbb{R}^2))}^{p-1}
	\int_0^t
		\Bigl( \bigl\| (-\Delta)^{\frac{\alpha-1}2} (x \partial_k \theta) \bigr\|_{L^p (\mathbb{R}^2)}
		+ \bigl\| (-\Delta)^{\frac{\alpha-2}2} \partial_k \theta \bigr\|_{L^p (\mathbb{R}^2)} \Bigr)
	ds\\
	&\le
	C \sum_{k=1}^2 \bigl\| |x|^2 \partial_k \theta \bigr\|_{L^\infty (0,t: L^p (\mathbb{R}^2))}^{p-1}
	\int_0^t
		\Bigl( \bigl\| x \partial_k \theta \bigr\|_{L^r (\mathbb{R}^2)} + \bigl\| \theta \bigr\|_{L^r (\mathbb{R}^2)} \Bigr)
	ds.
\end{split}
\]
Since $\frac{pr}{2p-r} > 1$, we have from the Sobolev inequality with \eqref{wttheta} that
\[
	\bigl\| x \partial_k \theta \bigr\|_{L^r (\mathbb{R}^2)}
	\le
	C \bigl\| \partial_k \theta \bigr\|_{L^{\frac{pr}{2p-r}} (\mathbb{R}^2)}^{1/2} \bigl\| |x|^2 \partial_k \theta \bigr\|_{L^p (\mathbb{R})}^{1/2}
	\le
	C (1+t)^{-\frac1{2\alpha}-1+\frac1{\alpha p}} \bigl\| |x|^2 \partial_k \theta \bigr\|_{L^p (\mathbb{R})}^{1/2}.
\]
Therefore
\begin{equation}\label{bsg4}
\begin{split}
	&\sum_{k=1}^2 \biggl|
		\int_0^t \int_{\mathbb{R}^2}
		\bigl| |x|^2 \partial_k \theta \bigr|^{p-2} |x|^2 \partial_k \theta \bigl[ (-\Delta)^{\alpha/2}, |x|^2 \bigr] \partial_k \theta
	dxds \biggr|\\
	&\le
	C \sum_{k=1}^2 \bigl\| |x|^2 \partial_k \theta \bigr\|_{L^\infty (0,t: L^p (\mathbb{R}^2))}^{p-\frac12}
	\int_0^t (1+s)^{-\frac1{2\alpha}-1+\frac1{\alpha p}} ds\\
	&+
	C \sum_{k=1}^2 \bigl\| |x|^2 \partial_k \theta \bigr\|_{L^\infty (0,t: L^p (\mathbb{R}^2))}^{p-1}
	\int_0^t (1+s)^{-\frac2\alpha (1-\frac1p)-1+\frac1\alpha} ds
	\le
	C_\delta + \delta  \sum_{k=1}^2 \bigl\| |x|^2 \partial_k \theta \bigr\|_{L^\infty (0,t: L^p (\mathbb{R}^2))}^p.
\end{split}
\end{equation}
Consequently, applying \eqref{bsg2}-\eqref{bsg4} to \eqref{bsg},
\begin{equation}\label{tokyu}
	 \bigl\| |x|^2 \nabla \theta \bigr\|_{L^\infty (0,\infty: L^p (\mathbb{R}^2))}^p
	 +
	 \int_0^\infty \bigl\| (-\Delta)^{\alpha/4} \bigl( \bigl| |x|^2 \nabla \theta \bigr|^{p/2} \bigr) \bigr\|_{L^2 (\mathbb{R}^2)}^2 ds
	 < + \infty.
\end{equation}
We back to the estimate for the last term on \eqref{bss}.
Since $\| |x|^2 \nabla^\bot \psi \|_{L^\infty (\mathbb{R}^2)} \le C$, this term fulfills that
\[
\begin{split}
	&\int_0^t \int_{\mathbb{R}^2}
		\bigl| G_\alpha (t-s,x-y) \bigr|
		\bigl| |y|^{2+\alpha} y_j (\nabla \theta \cdot \nabla^\bot\psi) (s,y) \bigr|
	dyds\\
	&\le
	C\int_0^t
		\bigl\| G_\alpha (t-s) \bigr\|_{L^{p'} (\mathbb{R}^2)}
		\bigl\| |x|^\alpha x_j \nabla \theta \bigr\|_{L^p (\mathbb{R}^2)}
	ds.
\end{split}
\]
Here \eqref{drbdd} and \eqref{tokyu} yield that
\[
	\bigl\| |x|^\alpha x_j \nabla \theta \bigr\|_{L^p (\mathbb{R}^2)}
	\le
	\bigl\| |x|^{\alpha + 1} \nabla \theta \bigr\|_{L^p (\mathbb{R}^2)}
	\le
	\bigl\| |x|^2 \nabla \theta \bigr\|_{L^p (\mathbb{R}^2)}^{\frac{1+\alpha}2}
	\bigl\| \nabla \theta \bigr\|_{L^p (\mathbb{R}^2)}^{\frac{1-\alpha}2}
	\le
	C (1+t)^{- (\frac2\alpha (1-\frac1p)+\frac1\alpha)\frac{1-\alpha}2}.
\]
Concurrently, if we choose $p_1$ sufficiently near from $p$, then for $\frac1r = \frac1{p_1} - \frac1p + \frac12$,
\[
	\bigl\| (-\Delta)^{-1/2} (x\theta) \bigr\|_{L^{\frac{p_1 p}{p-p_1}} (\mathbb{R}^2)}
	\le
	C \bigl\| x\theta \bigr\|_{L^r (\mathbb{R}^2)}
	\le
	C \bigl\| |x|^2 \theta \bigr\|_{L^p (\mathbb{R}^2)}^{1/2} \bigl\| \theta \bigr\|_{L^{\frac{pr}{2p-r}} (\mathbb{R}^2)}^{1/2}
	\le C (1+t)^{\frac2\alpha (\frac1{p_1} - \frac1p)}.
\]
Therefore
\[
\begin{split}
	&\int_0^t \int_{\mathbb{R}^2}
		\bigl| G_\alpha (t-s) \bigr|
		\bigl| |x|^{2+\alpha} x_j \nabla \theta \cdot \nabla^\bot\psi (s) \bigr|
	dyds
	\le
	C \int_0^t
		(t-s)^{-\frac2{\alpha p}}
		(1+s)^{- (\frac2\alpha (1-\frac1p)+\frac1\alpha)\frac{1-\alpha}2}
	ds\\
	&+
	C \int_0^t
		(t-s)^{-\frac2{\alpha p_1}}
		(1+s)^{- (\frac2\alpha (1-\frac1p)+\frac1\alpha)\frac{1-\alpha}2+\frac2\alpha (\frac1{p_1}-\frac1p)}
	ds
	\le
	C (1+t).
\end{split}
\]
Similarly, for the second term on \eqref{bss}, we obtain that
\[
\begin{split}
	&\int_0^t \int_{\mathbb{R}^2}
		\bigl|  (x_j-y_j) \nabla G_\alpha (t-s,x-y) \bigr|
		\bigl| |y|^{2+\alpha} (\theta \nabla^\bot \psi)(s,y) \bigr|
	dyds\\
	&\le
	C \int_0^t
		\bigl\| x_j \nabla G_\alpha (t-s) \bigr\|_{L^{p'} (\mathbb{R}^2)} \bigl\| |x|^\alpha \theta \bigr\|_{L^p (\mathbb{R}^2)} 
	ds\\
	&\le
	C \int_0^t
		(t-s)^{-\frac2{\alpha p}} (1+s)^{-\frac2\alpha (1-\frac1p)+1}
	ds
	\le
	C (1+t).
\end{split}
\]
The fourth term on \eqref{bss} also is bouded by $C (1+t)$.
For the fifth term on \eqref{bss}, we see from \eqref{tokyu} that
\[
\begin{split}
	&\int_0^t
		\bigl\| |x|^{2+\alpha} G_\alpha (t-s) \bigr\|_{L^\infty (\mathbb{R}^2)}
		\bigl\| x_j (\nabla \theta \cdot \nabla^\bot\psi) (s) \bigr\|_{L^1 (\mathbb{R}^2)}
	ds\\
	&\le
	\int_0^t
		\bigl\| |x|^{2+\alpha} G_\alpha (t-s) \bigr\|_{L^\infty (\mathbb{R}^2)}
		\bigl\| x_j^2 \nabla \theta \bigr\|_{L^p (\mathbb{R}^2)}^{1/2}
		\bigl\| \nabla \theta \bigr\|_{L^p (\mathbb{R}^2)}^{1/2}
		\bigl\|  \nabla^\bot\psi \bigr\|_{L^{p'} (\mathbb{R}^2)}
	ds\\
	&\le
	C \int_0^t
		(t-s) (1+s)^{-\frac3{2\alpha}-\frac1{\alpha p}}
	ds
	\le
	C t.
\end{split}
\]
Applying those estimates to \eqref{bss}, we complete the proof.
\end{proof}

Finally, Lemma \ref{lem1} and Proposition \ref{prop-v2} show Theorem \ref{thm}.
Also we conclude Corollary \ref{cor} from \eqref{B-G}, Lemma \ref{lem2} and Proposition \ref{prop-v2}.

\end{document}